\documentclass[reqno, 11pt]{amsart}
\usepackage{amsmath}
 \usepackage{amssymb}
\usepackage{amsthm}
\usepackage{times}
\usepackage{latexsym}
\usepackage[mathscr]{eucal}

\numberwithin{equation}{section}
 
  \newtheorem{theorem}{Theorem}[section]
  \newtheorem{proposition}[theorem]{Proposition}
  \newtheorem{lemma}[theorem]{Lemma}

  \newtheorem{definition}[theorem]{Definition}

\title[Induced and intrinsic Hashiguchi connections on Finsler submanifolds ]{Induced and intrinsic Hashiguchi connections on Finsler submanifolds }
\author[ Fortun\'{e} Massamba and Salomon Joseph Mbatakou]{ Fortun\'{e} Massamba*, \, Salomon Joseph Mbatakou** }

\newcommand{\acr}{\newline\indent}

\address{\llap{*\,}  School of Mathematics, Statistics and Computer Science\acr
 University of KwaZulu-Natal\acr
 Private Bag X01, Scottsville 3209\acr
South Africa}
\email{massfort@yahoo.fr, Massamba@ukzn.ac.za}

\address{\llap{**\,}  School of Mathematics, Statistics and Computer Science\acr
 University of KwaZulu-Natal\acr
 Private Bag X01, Scottsville 3209\acr
South Africa}
\email{mbatakou@gmail.com, mbatakous@ukzn.ac.za}

\thanks{}

\subjclass[2010]{Primary 53D10; Secondary 53C60}

\keywords{ Hashiguchi connection; Finslerian submanifold; Pulled-back Finsler connection.}

\begin{document}
\maketitle

 \begin{abstract}
We study the geometry of Finsler submanifolds using the pulled-back approach.  We define the Finsler normal pulled-back bundle and obtain the induced geometric objects, namely, induced pullback Finsler  connection, normal pullback Finsler connection, second fundamental form and shape operator. Under a certain condition, we prove that induced and intrinsic Hashiguchi connections coincide on the pulled-back bundle of Finsler submanifold. 
 \end{abstract}

 \section {Introduction}

Let $(\widetilde {M},\widetilde {F})$ be a Finsler manifold and $T\widetilde {M}_{0}$ be its slit tangent bundle. There exist in the literature, several frameworks for the study of Finsler geometry. For example an approach through the double tangent bundle $TT\widetilde {M}_{0}$ (Grifone's approach, \cite{GR}), an approach via the vertical subbundle of $TT\widetilde {M}_{0}$ (see Bejancu-Farran \cite{bl}) and the pulled-back bundle approach (see Bao-Chern-Shen \cite{AA}). The latter is for us the most natural approach, because it facilitates the analogy with the Riemannian geometry. 
In \cite{bl} Bejancu and Farran described the theory of Finsler submanifold via the vertical bundle and applied their study to some induced geometric objects as connections and curvatures. The study was applied for the induced and intrinsic Cartan, Chern and Berwald connections, but in Finsler geometry there is also a famous connection, namely Hashiguchi connection which also deserves to be studied. The purpose of the present paper is to suggest under the  pulled-back approach of Finsler submanifold, a comparison between the induced and the intrinsic Hashiguchi connections on the Finsler submanifold. The paper consists of three sections. In Section \ref{section1} we provide a brief account of the basic definitions and concepts that are used throughout the paper. For more details we refer to \cite{bl}, \cite{mb1}, \cite{mb2}. In Section \ref{NormTang},  we construct the Finsler normal and tangential pulled-back bundle and obtain the induced geometric objects: Induced and intrinsic Finsler-Ehresmann connection, the pulled-back Finsler connection, the second fundamental form and the shape operator. Finally, the Section \ref{InduIntri} is devoted to the comparison between the induced and the intrinsic Hashiguchi connection on Finsler submanifold.

 \section{Preliminaries}\label{section1}

Let $\pi: TM \rightarrow M $ be a tangent bundle of a connected smooth Finsler manifold $M$ of dimension $m$.
We denote by $v= (x,y)$ the points in $TM$ if $y\in \pi^{-1}(x)= T_{x}M$. We denote by $O(M)$ the zero section of $TM$,
and by $TM_{0}$ the slit tangent bundle $ TM\setminus O(M)$. We introduce a coordinate system on $TM$ as follows.
Let $U\subset M$ be an open set with local coordinate $(x^{1},...,x^{m})$.
By setting $v= y^{i}\frac{\partial}{\partial x^{i}}$ for every $v\in\pi^{-1}(U) $, we introduce a
local coordinate $(x,y)=(x^{1},...,x^{m},y^{1},...,y^{m})$ on $\pi^{-1}(U)$.

\begin{definition}{\rm
 A function $F: TM \rightarrow [0,+\infty[$ is called a Finsler structure or Finsler metric on $M$ if:
  \begin{enumerate}
  \item[(i)] $F\in C^{\infty}(TM_{0})$,
\item[(ii)]  $F(x,\lambda y)= \lambda F(x,y)$, for all $ \lambda > 0$,
\item[(iii)]  The $m\times m$ Hessian matrix $( g_{ij})$, where
\begin{equation}
\label{ft} g_{ij}:= \frac{1}{2}(F^{2})_{y^{i}y^{j}}
\end{equation}
  is positive-definite at all $(x,y)$ of $TM_{0}$.
  \end{enumerate}
  }
\end{definition}
The pair $(M, F )$ is called \textit{Finsler manifold}. The pulled-back bundle  $\pi^{\ast}TM$ is a vector bundle over the slit tangent bundle $ TM_{0}$, defined by
\begin{equation}
\pi^{\ast}TM:= \{(x,y,v)\in TM_{0}\times TM: v \in T_{\pi(x,y)}M \}.
\end{equation}
 By (\ref{ft}), the pulled-back vector bundle  $\pi^{\ast}TM$ admits a natural Riemannian metric
 \begin{equation}\label{tefun}
 g:= g_{ij}dx^{i}\otimes dx^{j}.
 \end{equation}
 This is in general called the \textit{fundamental tensor} (see \cite{AA} for more details). Likewise, there are some Finslerian tensors which play  important roles in the Finslerian geometry, namely, the distinguished section 
 \begin{equation}\label{DistSect}
 l= \frac{y^{i}}{F}\frac{\partial}{\partial x^{i}}, 
 \end{equation}
 and the Cartan tensor given by
 \begin{equation}\label{Atensor}
 A= A_{ijk}dx^{i}\otimes dx^{j}\otimes dx^{k},
 \end{equation}
 where
 \begin{equation}
 A_{ijk}:= \frac{F}{2}\frac{\partial g_{ij}}{\partial y^{k}}.
 \end{equation}
Note that, with a slight abuse of notation, $\frac{\partial}{\partial x^{i}}$ and $dx^{i}$ are regarded as sections
 of $\pi^{\ast}TM$ and $\pi^{\ast}T^{\ast}M$, respectively.

Now, for the differential $\pi_{\ast}$ of the submersion $\pi: TM_{0}\rightarrow M$, the vertical subbundle $\mathcal{V}$ of
$TTM_{0}$ is defined by $\mathcal{V}= \ker \pi_{\ast}$, and $\mathcal{V}$ is locally spanned by $\{F\frac{\partial}
{\partial y^{1}},...,F\frac{\partial}{\partial y^{n}}\}$ on each $\pi^{-1}(U)$. Then, it induces the exact sequence
 \begin{equation}\label{sui}
 0\longrightarrow \mathcal{V}  \stackrel{i}{\longrightarrow}  TTM_{0}\stackrel{\pi_{\ast}}{\longrightarrow}
 \pi^{\ast}TM \longrightarrow 0,
 \end{equation}
The horizontal subbundle $\mathcal{H}$ is
defined by a subbundle $\mathcal{H}\subset TTM_{0}$, which is complementary to $\mathcal{V}$. These subbundles
give a smooth splitting
\begin{equation}\label{dec}
TTM_{0} =   \mathcal{H} \oplus   \mathcal{V}.
\end{equation}
Although the vertical subbundle $ \mathcal{V}$ is uniquely determined, the horizontal subbundle is not canonically
determined. An Ehresmann connection of the submersion $\pi:TM_{0}\rightarrow M$ depends on a choice of horizontal subbundles.

In this paper, we shall consider the choice of Ehresmann connection which arises from the Finsler structure $F$,
constructed as follows. Recall that \cite{AF} every Finslerian structure $F$ induces a spray
$$
G= y^{i}\frac{\partial}{\partial x^{i}}- 2 G^{i}(x,y)\frac{\partial}{\partial y^{i}},
$$
in which the spray coefficients  $G^{i}$ are defined by
\begin{equation}
 G^{i}(x,y):= \frac{1}{4}g^{il}\left[2 \frac{\partial g_{jl}}{\partial x^{k}}(x,y)-
 \frac{\partial g_{jk}}{\partial x^{l}}(x,y)\right]y^{j}y^{k},
\end{equation}
where the matrix $(g^{ij})$ means the inverse of $(g_{ij})$.

Define a $\pi^{\ast}TM$-valued smooth form on $TM_{0}$ by
\begin{equation}\label{teta}
\theta = \frac{\partial}{\partial x^{i}}\otimes \frac{1}{F}(dy^{i}+N_{j}^{i}dx^{j}),
\end{equation}
where functions  $N^{i}_{j}(x,y)$ are given by
$$
N^{i}_{j}(x,y):= \frac{\partial G^{i}}{\partial y^{j}}(x,y).
$$
This $\pi^{\ast}TM$-valued smooth form  $\theta$ is globally well defined on $TM_{0}$ \cite{AC}.

From the form $\theta$ defined in (\ref{teta}) which is called \textit{Finsler-Ehresmann form}, we define the \textit{Finsler-Ehresmann connection} as follow.
\begin{definition}{\rm
 A Finsler-Ehresmann connection of the submersion $\pi:TM_{0}\rightarrow M$ is the subbundle $\mathcal{H}$ of $TTM_{0}$
 given by $ \mathcal{H} = \ker \theta $, where $\theta : TTM_{0}\rightarrow \pi^{\ast}TM$ is the bundle morphism defined in  (\ref{teta}),
 and which is complementary to the vertical subbundle $\mathcal{V}$.}
\end{definition}
It is well know that, $\pi^{\ast}TM$ can be naturally identified with the horizontal subbundle $\mathcal{H}$ and the vertical one
$\mathcal{V}$ \cite{AE}.
Thus, any section $\overline{X}$ of $\pi^{\ast}TM$ is considered as a section of $\mathcal{H}$ or a section of $\mathcal{V}$.
We denote by $\overline{X}^{H}$ and $\overline{X}^{V}$ respectively, the section
of $\mathcal{H}$ and the section of $\mathcal{V}$ corresponding to $\overline{X}\in \Gamma (\pi^{\ast}TM)$:
\begin{equation}\label{tr}
 \overline{X}= \frac{\partial}{\partial x^{i}}\otimes \overline{X}^{i}\in \pi^{\ast}TM
\Longleftrightarrow \overline{X}^{H}=\frac{\delta}{\delta x^{i}}\otimes \overline{X}^{i} \in \Gamma(\mathcal{H}),
\end{equation}
and
\begin{equation}
\overline{X}= \frac{\partial}{\partial x^{i}}\otimes \overline{X}^{i}\in \pi^{\ast}TM
\Longleftrightarrow \overline{X}^{V}=F\frac{\partial}{\partial y^{i}}\otimes \overline{X}^{i} \in \Gamma(\mathcal{V}),
\end{equation}
where
$$
\{F\frac{\partial}{\partial y^{i}}:= (\frac{\partial}{\partial x^{i}})^{V}\}_{i= 1,...,m}\;\;\mbox{and}\;\;\{\frac{\delta}{\delta x^{i}}:=\frac{\partial}{\partial x^{i}} - N_{j}^{i}\frac{\partial}{\partial y^{i}} = (\frac{\partial}{\partial x^{i}})^{H}\}_{i= 1,...,m},
$$
are the vertical and horizontal lifts of natural local frame field $\{ \frac{\partial}{\partial x^{1}},..., \frac{\partial}{\partial x^{m}}\}$ with respect to the Finsler-Ehresmann connection $\mathcal{H}$, respectively.

Let $\{dx^{1},...,dx^{m}\}$ be the basis of the dual space $\mathcal{H}^{\ast}$, and $\{\frac{\delta y^{i}}{F}:= \frac{1}{F}(dy^{i} + N^{i}_{j}dx^{j})\}_{i= 1,...,m}$ the basis of $\mathcal{V}^{\ast}$. For two bundle morphisms $\pi_{\ast}$ and $\theta$  from $TTM_{0}$ onto $\pi^{\ast}TM$, we have the following.
\begin{proposition} \label{PropIso}\cite{AE}
 The bundle morphism $\pi_{\ast}$ and $\theta$ satisfy
 $$
 \pi_{\ast}(\overline{X}^{H}) = \overline{X}, \;\; \pi_{\ast}(\overline{X}^{V})= 0,\;\; \theta(\overline{X}^{H})  = 0,\;\;\theta(\overline{X}^{V})= \overline{X},
 $$ 
for every $\overline{X} \in \Gamma(\pi^{\ast}TM)$.
\end{proposition}
The Proposition \ref{PropIso} means that $\mathcal{H} T M_{0}$, as well as $\mathcal{V} T M_{0}$, can be naturally identified with the bundle $\pi^{*} T M$, that is,
\begin{equation}\label{Isomorphs}
 \mathcal{H} T M_{0}\cong \pi^{*} T M\;\;\mbox{and}\;\; \mathcal{V} T M_{0}\cong \pi^{*} T M.
\end{equation}

 \subsection{Pulled-back bundle Finsler connection.} 
Now, by the Finsler-Ehresmann connection and any linear connection on the pulled-back bundle $\pi^{*} T M$, we introduce the concept of Pulled-back bundle Finsler connection. 
\begin{definition}
 Let $(M,F)$ be a Finsler manifold and $\pi^{\ast}TM$ the pulled-back bundle over $TM_{0}$. Suppose that there exist a linear connection $\nabla$ on 
 $\pi^{\ast}TM$ and Finsler-Ehresmann form $\theta$ on $TM_{0}$. Then the pulled-back bundle Finsler connection on $\pi^{\ast}TM$ is the pair $(\ker{\theta},\nabla)$.
\end{definition} 
\noindent
Note that all outstanding connections of the Finsler geometry, namely \cite{mb1}: Cartan, Berwald, Chern and Hashiguchi connections on $\pi^{\ast}TM$ are the pulled-back bundle Finsler connections.

\section{ Normal and tangential Finsler Pulled-Back Bundle}\label{NormTang}

Let $(\widetilde {M},\widetilde {F})$ be a real $(m+n)$- dimensional Finsler manifold. By (\ref{tefun}), there exist a Riemannian metric
$\widetilde {g}$ on the pulled-back bundle $\pi^{\ast}T\widetilde {M}$, whose local components are given by
\begin{equation}
 \widetilde {g}_{ij}(x,y)= \frac{1}{2}\frac{\partial^{2}\widetilde {F}^{2}}{\partial y^{i}\partial y^{j}}.
\end{equation}
The local coordinates on $T\widetilde {M}_{0}$ will be $(x^{i}, y^{i})$, $i\in \{1,\cdots,m+n\}$, where $(x^{i})$ are 
the local coordinates on $\widetilde {M}$.

 In the following, we use the ranges for indices: $i,j,k,\cdots \in \{1,\cdots, m+n\}$; 
 $\alpha, \beta, \gamma,\cdots \in \{1,\cdots, m\}$; $a,b,c \cdots \in \{m+1, \cdots, m+n\}$.
  
  Suppose that $M$ is a real $m$-dimensional submanifold of $\widetilde {M}$ defined by the equations 
  \begin{equation}\label{funsub}
   x^{i}=x^{i}(u^{1},\cdots, u^{m}); \quad rank[B^{i}_{\alpha}]= m;\quad B^{i}_{\alpha}= \frac{\partial x^{i}}{\partial u^{\alpha}}.
  \end{equation}
Denote by $\iota$ the immersion of $M$ in $\widetilde {M}$ and consider the tangent map $\iota_{\ast}$ of $TM_{0}$ in $T\widetilde {M}_{0}$. Locally,
a point of $TM_{0}$ with coordinates $(u^{\alpha}, v^{\alpha})$ is carried by $\iota_{\ast}$ into a point of $T\widetilde {M}_{0}$ with coordinates
$(x^{i}(u), y^{i}(u,v))$, where $x^{i}$ are function in (\ref{funsub}) and 
\begin{equation}
 y^{i}(u,v)= B^{i}_{\alpha}v^{\alpha}.
\end{equation}
Recall that the sections $\frac{\partial}{\partial x^{i}}$ and $dx^{i}$ of $T\widetilde {M}$ and his dual $T^{\ast}\widetilde {M}$ give rise to sections 
of the pulled-back bundles \cite{AA}. In order to keep the notation simple, we also use the symbols
$\frac{\partial}{\partial x^{i}}$ and $dx^{i}$ to denote the basis sections of $\pi^{\ast}T\widetilde {M}$ and $\pi^{\ast}T^{\ast}\widetilde {M}$.

Note that, the pulled-back bundle $\pi^{\ast}TM$ is a vector subbundle of  $\pi^{\ast}T\widetilde {M}|_{TM_{0}}$. In fact for all $(u,v)\in TM_{0}$,
$\pi^{\ast}TM|_{(u,v)}$ is a  vector subspace of $\pi^{\ast}T\widetilde {M}|_{(u,v)}$. Furthermore, the relation 
\begin{equation}
 \frac{\partial}{\partial u^{\alpha}}= B^{i}_{\alpha}\frac{\partial}{\partial x^{i}},
\end{equation}
 between the basis  of $T_{u}M$ and that of $T_{u}\widetilde {M}$, induces, by natural lift, a similar relation between the basis sections $ \frac{\partial}{\partial u^{\alpha}}$ and $\frac{\partial}{\partial x^{i}}$ of $\pi^{\ast}TM$ and $\pi^{\ast}T\widetilde {M}|_{TM_{0}}$, respectively. 
Hence the Riemannian metric $\widetilde {g}$ induces a Riemannian metric $g$ on $\pi^{\ast}TM$. More precisely, $g$ is locally defined by  
\begin{equation}
 g_{\alpha\beta}(u,v)= B_{\alpha}^{i}B^{j}_{\beta}\widetilde {g}_{ij}(x(u),y(u,v)).
\end{equation}
On the other hand, the Finsler structure $\widetilde {F}$ of $\widetilde {M}$ induces on $TM_{0}$ the function $F$ locally given by
\begin{equation}\nonumber
 F(u,v)= \widetilde {F}(x(u), y(u,v))
\end{equation}
 Then by straightforward calculations it follows that $(M,F)$ is a Finsler manifold and the fundamental tensor of $F$ is given by
\begin{equation}
 g_{\alpha\beta}=\frac{1}{2}\frac{\partial^{2}F^{2}}{\partial v^{\alpha}\partial v^{\beta}}.
\end{equation}
Let denote by $\pi^{\ast}TM^{\bot}$ the orthogonal complementary pulled-back bundle of $\pi^{\ast}TM$ in $\pi^{\ast}T\widetilde {M}|_{TM_{0}}$
with respect to $\widetilde {g}$. The bundle $\pi^{\ast}TM^{\bot}$ is called normal pulled-back bundle of Finsler submanifold $(M,F)$ in $(\widetilde {M},\widetilde {F})$. So we have the orthogonal
decomposition 
\begin{equation}\label{ortd}
 \pi^{\ast}T\widetilde {M}|_{TM_{0}}= \pi^{\ast}TM \oplus \pi^{\ast}TM^{\bot}.
\end{equation}
Now, we consider local sections of orthonormal  basis $\{ \mathfrak{N}_{a}= \mathfrak{N}_{a}^{i}\frac{\partial}{\partial x^{i}}\}$ 
of $\pi^{\ast}TM^{\bot}$ with respect to $\widetilde {g}$, where  $\mathfrak{N}^{i}_{a}$ are the functions on $TM_{0}$, satisfying the following relations:
\begin{equation}
 \widetilde {g}_{ij}B^{i}_{\alpha}\mathfrak{N}_{a}^{j}= 0; \quad \forall \alpha \in \{1,\cdots, m\}, \quad 
 a\in \{m+1,\cdots, m+n\},
\end{equation}
and
\begin{equation}
 \widetilde {g}_{ij}\mathfrak{N}^{i}_{a}\mathfrak{N}_{b}^{j}= \delta_{ab}; \quad \forall a,b \in \{m+1,\cdots, m+n\}.
\end{equation}
Denote by $\left[B^{i}_{\alpha} \mathfrak{N}^{i}_{a}\right]$ the transition matrix from $\{\frac{\partial}{\partial x^{i}}\}$ to 
$\{\frac{\partial}{\partial u^{\alpha}}, \mathfrak{N}_{a}\}$ and, $\left[\widetilde {B}^{\alpha}_{i} \widetilde {\mathfrak{N}}^{a}_{i}\right]$ his
inverse. We have 

 \begin{align}
  &\widetilde {B}^{\alpha}_{i} B^{i}_{\beta}  = \delta^{\alpha}_{\beta};\quad \widetilde {B}^{\alpha}_{i} \mathfrak{N}^{i}_{a}= 0, \quad 
  \widetilde {\mathfrak{N}}_{i}^{a}B^{i}_{\alpha}=0,\quad  \widetilde {\mathfrak{N}}_{i}^{a}\mathfrak{N}^{i}_{b}= \delta^{a}_{b},\\\label{rsmj}
 \mbox{and}\;\; &B^{i}_{\alpha} \widetilde {B}^{\alpha}_{j} + \widetilde {\mathfrak{N}}_{j}^{a}\mathfrak{N}^{i}_{a}  = \delta_{j}^{i}.
\end{align}
In the sequel we use the notations: $ B^{ij...}_{\alpha\beta...}= B^{i}_{\alpha}B^{j...}_{\beta...}$ and $B^{i}_{\alpha0}= B^{i}_{\alpha\beta}v^{\beta}$.

\subsection{Induced and Intrinsic Finsler-Ehresmann connection.}
Recall that  the kernel of $\widetilde {\theta}$ represents the Finsler-Ehresmann connection, and  its orthogonal complementary distribution of $VT\widetilde {M}_{0}$ in $TT\widetilde {M}_{0}$ with respect to  the Sasaki-type metric on $T\widetilde {M}_{0}$,  given  by 
\begin{equation}
 \widetilde {G}_{s}= \widetilde {g}_{ij}dx^{i} \otimes dx^{j}+ \widetilde {g}_{ij}\frac{\delta y^{i}}{\widetilde {F}} \otimes \frac{\delta y^{j}}{\widetilde {F}}.
\end{equation}
By $\widetilde {G}_{s}$ we obtain the  induced Finsler-Ehresmann form  $\theta$, defined by 
 $$
 \theta:= \frac{\delta v^{\alpha}}{F}\otimes\frac{\partial}{\partial u^{\alpha}},
 $$
 where $\frac{\delta v^{\alpha}}{F}= \frac{1}{F}(dv^{\alpha}+ N^{\alpha}_{\beta}du^{\beta})$ and $N^{\alpha}_{\beta}$ is related to
$\widetilde {N}^{i}_{j}$ by (see \cite{Ru}, \cite{bl}):
\begin{equation}\label{ehi}
 N^{\alpha}_{\beta}= \widetilde {B}^{\alpha}_{i}(B^{i}_{\beta 0}+ B^{j}_{\beta}\widetilde {N}^{i}_{j}).
\end{equation}
On the other hand, the coefficients $N^{\alpha}_{\beta}$  has an intrinsic analogous $\hat{N}^{\alpha}_{\beta}$, obtained by the spray coefficients of $F$ on $M$, and related to $N^{\alpha}_{\beta}$, by 
\begin{equation}\label{inef}
 \hat{N}^{\alpha}_{\beta}=N^{\alpha}_{\beta} + \frac{A^{\alpha}_{\beta a}}{F}H^{a}_{\lambda}v^{\lambda},
\end{equation}
where the functions $ H^{a}_{\lambda}$ and $A^{\alpha}_{\beta a}$ are given respectively by
\begin{equation}\label{hal}
 H^{a}_{\lambda}= \widetilde {\mathfrak{N}}^{a}_{k}\left(B^{k}_{\lambda 0}+ B^{i}_{\lambda}\widetilde {N}^{k}_{i}\right),\quad \hbox{and}\quad A^{\alpha}_{\beta a}= g^{\alpha \lambda}A_{\lambda\beta a}.
\end{equation}
By (\ref{inef}), we obtain the intrinsic Finsler-Ehresmann $\pi^{\ast}TM$-valued form $\hat{\theta}$ given by:
\begin{equation}
 \hat{\theta}= \frac{1}{F}(dv^{\alpha}+ \hat{N}^{\alpha}_{\beta}du^{\beta})\otimes \frac{\partial}{\partial u^{\alpha}}.
\end{equation}
\begin{proposition}\label{cnl}
 Let $\widetilde {\theta}$ be the Finsler-Ehresmann $\pi^{\ast}T\widetilde {M}$-valued form on $T\widetilde {M}_{0}$. Then the induced Finsler-Ehresmann $\pi^{\ast}TM$-valued form $\theta$ coincides with $\widetilde {\theta}$ on $TM_{0}$. Moreover the intrinsic Finsler-Ehresmann form $\hat{\theta}$, and the induced one $\theta$ are related by
 \begin{equation}
  \hat{\theta}= \theta+ \frac{D}{F^{2}},
 \end{equation}
where $D$ is the $(0,1;1)$-tensor called deformation tensor and given by,
\begin{equation}\label{Deforma}
 D= D^{\alpha}_{\beta} du^{\beta}\otimes \frac{\partial}{\partial u^{\alpha}}= H^{a}_{\lambda}v^{\lambda}A^{\alpha}_{\beta a}du^{\beta}\otimes \frac{\partial}{\partial u^{\alpha}}.
\end{equation}
\end{proposition}
\begin{proof}
 We have
 \begin{eqnarray}
  \frac{\delta y^{i}}{\widetilde {F}}&=&\frac{1}{\widetilde {F}}(dy^{i}-\widetilde {N}^{i}_{j}dx^{j})\cr
  &=& \frac{1}{\widetilde {F}}\left(B^{i}_{\alpha 0}du^{\alpha}+ B^{i}_{\alpha}dv^{\alpha}+ \widetilde {N}^{i}_{j}B^{j}_{\alpha}du^{\alpha}\right)\cr
  &=& \frac{1}{\widetilde {F}}\left(B^{i}_{\beta}N^{\beta}_{\alpha}du^{\alpha}+ B^{i}_{\beta}dv^{\beta}\right) = B^{i}_{\beta}\frac{\delta v^{\beta}}{F}
 \end{eqnarray}
Elsewhere  $\frac{\partial}{\partial x^{i}}= \widetilde { B}^{\alpha}_{i}\frac{\partial}{\partial u^{\alpha}}+ \widetilde {\mathfrak{N}}^{a}_{i}\mathfrak{N}_{a}$, as basis sections of $\pi^{\ast}T\widetilde {M}|_{TM_{0}}$, it follows
that $\widetilde {\theta}= \theta$ on $TM_{0}$, since $B^{i}_{\beta}\widetilde {\mathfrak{N}}^{a}_{i}=0$. Also, we have  
\begin{align}
           \hat{\theta}&=  \frac{1}{F}(dv^{\alpha} + \hat{N}^{\alpha}_{\beta}du^{\beta})\otimes \frac{\partial}{\partial u^{\alpha}}\cr
           &= \frac{1}{F}\left(dv^{\alpha}+(N^{\alpha}_{\beta} + \frac{A^{\alpha}_{\beta a}}{F}H^{a}_{\lambda}v^{\lambda})du^{\beta} \right)\otimes \frac{\partial}{\partial u^{\alpha}} = \theta+ \frac{D}{F^{2}},
          \end{align}
which completes the proof.
\end{proof}
Note that the corresponding horizontal section $l^{H}$ of the distinguished section $l$ defined in (\ref{DistSect}) is given by $l^{H}= l^{\alpha}\frac{\delta}{\delta u^{\alpha}}$. 
\begin{lemma}\label{nD}
 The action of the deformation $(0,1;1)$-tensor $D$ defined in (\ref{Deforma}) on $l^{H}$ vanishes, that is 
 $
  D(l^{H})= 0.
 $
\end{lemma}
\begin{proof}
The proof follows from a direct calculations using the fact that the Cartan tensor $A$ vanishes along the distinguished section $l$.
\end{proof}
Let us denote by $\widetilde {\theta}$ the  Finsler-Ehresmann form  on $VT\widetilde {M}_{0}$. It is easy to see that $\widetilde {\theta}$  is a bundle isomorphism of $VT\widetilde {M}_{0}$ onto $\pi^{\ast}T\widetilde {M}$ (see \cite{AE} for more details). Therefore, we have the following.
\begin{lemma}\label{cnli}
 Let $(M,F)$ be a Finsler submanifold of  $(\widetilde {M}, \widetilde {F})$, $HT\widetilde {M}_{0}|_{TM_{0}}$ be a Finsler-Ehresmann connection, and $\pi^{\ast}TM^{\bot}$ be the normal pulled-back  bundle on $TM_{0}$. Then the induced Finsler-Ehresmann connection $HTM_{0}$ is a vector subbundle of 
 $HT\widetilde {M}_{0}|_{TM_{0}}\oplus \widetilde{\theta}^{-1}(\pi^{\ast}TM^{\bot})$.
\end{lemma}
\begin{proof}
Let $\widetilde {\theta}^{-1}$ the inverse of $\widetilde {\theta}$, we have 
 \begin{eqnarray}
  TT\widetilde {M}_{0}|_{TM_{0}} &=& HT\widetilde {M}_{0}|_{TM_{0}} \oplus VT\widetilde {M}_{0}|_{TM_{0}}\cr
  &=& HT\widetilde {M}_{0}|_{TM_{0}} \oplus \widetilde {\theta}^{-1}(\pi^{\ast}T\widetilde {M}|_{TM_{0}})\cr
  &=& HT\widetilde {M}_{0}|_{TM_{0}} \oplus \widetilde {\theta}^{-1}(\pi^{\ast}TM) \oplus \widetilde {\theta}^{-1}(\pi^{\ast}TM^{\bot})
 \end{eqnarray}
 Now, using the Proposition \ref{cnl} we have $\widetilde {\theta}^{-1}(\pi^{\ast}TM)= \theta^{-1}(\pi^{\ast}TM)= VTM_{0}$. It follows that
 $HT\widetilde {M}_{0}|_{TM_{0}} \oplus\widetilde {\theta}^{-1}(\pi^{\ast}TM^{\bot})$ is the orthogonal complementary vector bundle to $VTM_{0}$, in 
 $ TT\widetilde {M}_{0}|_{TM_{0}}$, and $HTM_{0}$ is orthogonal to $VTM_{0}$. We deduce that $HTM_{0}$ is a vector subbundle of 
 $HT\widetilde {M}_{0}|_{TM_{0}} \oplus\widetilde {\theta}^{-1}(\pi^{\ast}TM^{\bot})$, hence the result.
\end{proof}
\begin{definition}{\rm
 Let $HTM_{0}$ be a Finsler-Ehresmann connection on the Finsler manifold $(M,F)$, and $\pi^{\ast}TM$ the pulled-back bundle over the slit tangent bundle  $TM_{0}$. The pullback Finsler connection  is the pair $(HTM_{0},\nabla)$, where $\nabla$ is a linear connection on $\pi^{\ast}TM$.}
\end{definition}
Considering  the induced Finsler-Ehresmann connection on $\widetilde {M}$, we proceed with the study of the geometric objects induced by $\widetilde {\nabla}$ on $\pi^{\ast}TM$. Then according to the orthogonal decomposition (\ref{ortd}), the Gauss and Weingarten formulas are given by 
\begin{eqnarray}\label{gss}
 \widetilde {\nabla}_{X}\xi&=& \nabla_{X}\xi + \mathcal{S}(\xi, X),\\
 \label{wgt}
 \widetilde {\nabla}_{X}\eta &=&-\mathcal{A}_{\eta}X + \nabla^{\bot}_{X}\eta,
 \end{eqnarray}
where, $X \in \Gamma(TTM_{0})$, $\xi \in \Gamma(\pi^{\ast}TM)$, $\eta \in \Gamma( \pi^{\ast}TM^{\bot}) $, $\mathcal{S}(\xi,X) \in \Gamma(\pi^{\ast}TM^{\bot})$
and $\mathcal{A}_{\eta}X \in \Gamma(\pi^{\ast}TM)$. It is easy to check that, $\nabla$ and $\nabla^{\bot}$ are respectively the  linear 
connections on $\pi^{\ast}TM$ and $\pi^{\ast}TM^{\bot}$. Thus, with the induced Finsler-Ehresmann connection $HTM_{0}$, we can define a pullback Finsler connections on $\pi^{\ast}TM$ and $\pi^{\ast}TM^{\bot}$ which derive from $\widetilde{\nabla}$ and will called respectively the induced pullback Finsler connection and the induced normal pullback Finsler connection. Note that for any $\eta\in \Gamma(\pi^{\ast}TM^{\bot})$,  we call $\mathcal{A}_{\eta}$ and $\mathcal{S}$ the  shape operator and the second fundamental form, respectively,  and these are  the Finslerian tensors of type $(0,1;1)$ and $(1,1;1)$, respectively.

Now for any $X\in \Gamma(TTM_{0})$ we define the differential operator, 
\begin{eqnarray}
 \overline{\nabla}_{X}: \Gamma(\pi^{\ast}T\widetilde{M}|_{TM_{0}})&\longrightarrow & \Gamma(\pi^{\ast}T\widetilde{M}|_{TM_{0}})\cr
 \widetilde{\xi}    &\longmapsto&  \overline{\nabla}_{X}\widetilde{\xi}:= \widetilde{\nabla}_{X}\widetilde{\xi}. 
\end{eqnarray}

Clearly, $\overline{\nabla}$ is a linear connection on $\pi^{\ast}T\widetilde{M}|_{TM_{0}}$. So $(HTM_{0}, \overline{\nabla})$ defines the restriction of the pullback Finsler connection $\widetilde{\nabla}$ on $\pi^{\ast}T\widetilde{M}|_{TM_{0}}$. 

We now consider 
the local coefficients of $\widetilde{\nabla}$, $\nabla$ and $\nabla^{\bot}$ given respectively by
\begin{eqnarray}
 \widetilde{\nabla}_{\frac{\delta}{\delta x^{j}}}\frac{\partial}{\partial x^{i}}= \widetilde{\Gamma}_{i j}^{k}\frac{\partial}{\partial x^{k}},
 \quad \widetilde{\nabla}_{F\frac{\partial}{\partial y^{j}}}\frac{\partial}{\partial x^{i}} = \widetilde{\gamma}_{i j}^{k}\frac{\partial}{\partial x^{k}}.\\
 \nabla_{\frac{\delta}{\delta u^{\beta}}}\frac{\partial}{\partial u^{\alpha}}= \Gamma_{\alpha \beta}^{\lambda}\frac{\partial}{\partial u^{\lambda}},\quad \nabla_{F\frac{\partial}{\partial v^{\beta}}}\frac{\partial}{\partial u^{\alpha}}= \gamma_{\alpha \beta}^{\lambda}\frac{\partial}{\partial u^{\lambda}}.
\end{eqnarray}
and 
\begin{equation}
  \nabla^{\bot}_{\frac{\delta}{\delta u^{\alpha}}}\mathfrak{N}_{a}= \Gamma_{a \alpha }^{b}\mathfrak{N}_{b},\quad \nabla^{\bot}_{F\frac{\partial}{\partial v^{\alpha}}}\mathfrak{N}_{a}= \gamma_{a \alpha }^{b}\mathfrak{N}_{b},
\end{equation}
where  $\widetilde{\Gamma}_{i j}^{k}$ and $\widetilde{\gamma}_{i j}^{k}$ are the ``Christoffel symbols'' with respect $\frac{\delta}{\delta x^{i}}$ and $\widetilde{F}\frac{\partial}{\partial y^{i}}$, respectively and given by:
\begin{eqnarray}
 \widetilde{\Gamma}_{i j}^{k}&=& \frac{\widetilde{g}^{is}}{2}\left(\frac{\delta \widetilde{g}_{sj}}{\delta x^{k}}-\frac{\delta \widetilde{g}_{jk}}{\delta x^{s}}+ \frac{\delta \widetilde{g}_{ks}}{\delta x^{j}}\right), \\ \hbox{and}\quad
 \widetilde{\gamma}_{i j}^{k}&=& \frac{\widetilde{F}}{2}\widetilde{g}^{ks}\frac{\partial \widetilde{g}_{ij}}{\partial y^{j}}.
\end{eqnarray}

 We also define locally, the horizontal and vertical part of second fundamental form respectively by 
\begin{equation}
 \mathcal{S}(\frac{\delta}{\delta u^{\alpha}}, \frac{\partial}{\partial u^{\beta}})=\stackrel{h}{\mathcal{S}^{a}_{\alpha \beta}}\mathfrak{N}_{a} \quad \hbox{and}\quad
 \mathcal{S}(F\frac{\partial}{\partial u^{\alpha}}, \frac{\partial}{\partial u^{\beta}})=\stackrel{v}{\mathcal{S}^{a}_{\alpha \beta}}\mathfrak{N}_{a}.
\end{equation}

Likewise,  the horizontal and vertical part of shape operator with respect to the normal section $\mathfrak{N}_{a}$
are given respectively  by 

\begin{equation}
\mathcal{A}(\mathfrak{N}_{a},\frac{\delta}{\delta u^{\alpha}}) = \stackrel{h}{ \mathcal{A}^{\alpha}_{a\beta}}\frac{\partial}{\partial u^{\alpha}}\quad \hbox{and}\quad
\mathcal{A}(\mathfrak{N}_{a},F\frac{\partial}{\partial v^{\alpha}}) = \stackrel{v}{\mathcal{A}_{a \beta}^{\alpha}}\frac{\partial}{\partial u^{\alpha}}.
\end{equation}

By the lemma \ref{cnli}, it is easy to check that the local coefficients of $\overline{\nabla}$ and $\widetilde{\nabla}$ are related by:
\begin{eqnarray}\label{cih}
\overline{\Gamma}^{i}_{j \alpha}& =& B^{k}_{\alpha}\widetilde{\Gamma}^{i}_{j k}+ H^{a}_{\alpha}\mathfrak{N}^{k}_{a}\widetilde{\gamma}^{i}_{jk},\\
\label{civ}
\overline{\gamma}^{i}_{j \alpha} &=& \widetilde{\gamma}^{i}_{jk}B^{k}_{\alpha}.
\end{eqnarray}
where $H^{a}_{\alpha}$ is given by (\ref{hal}). 

 The local coefficients of the induced pullback Finsler connection $\nabla$ are given in terms of the local coefficients of $\widetilde{\nabla}$ by \cite{bl}:
 \begin{eqnarray}\label{coef}
  \Gamma_{\alpha \beta}^{\lambda}&=& \widetilde{B}_{i}^{\lambda}\left(B^{i}_{\alpha \beta}+ \widetilde{\Gamma}^{i}_{jk}B^{jk}_{\alpha \beta} + \widetilde{\gamma}_{jk}^{i}B^{j}_{\alpha}H^{a}_{\beta}\mathfrak{N}_{a}^{k}\right),\\
 \label{coefv} \gamma_{\alpha \beta}^{\lambda} &=& \widetilde{B}_{i}^{\lambda}\widetilde{\gamma}_{jk}^{i}B^{jk}_{\alpha \beta}.
 \end{eqnarray}
Likewise, the local coefficients of $\nabla^{\bot}$ are given in function of the local coefficients of $\widetilde{\nabla}$ by:
\begin{eqnarray}\label{coefnh}
 \Gamma_{a \alpha}^{b} &=& \left(\frac{\delta B^{k}_{a}}{\delta u^{\alpha}} + B^{j}_{a}B^{k}_{\alpha}\widetilde{\Gamma}^{i}_{jk}+ B^{j}_{a}H^{a}_{\alpha}\mathfrak{N}_{a}^{k}\widetilde{\gamma}_{jk}^{i}\right)\widetilde{\mathfrak{N}}^{b}_{i},\\
 \label{coefnv}
 \gamma_{a \alpha}^{b} &=& \left(F\frac{\partial B^{k}_{a}}{\partial v^{\alpha}}+ B^{j}_{a}\widetilde{\gamma}^{i}_{jk}B^{k}_{\alpha}\right)\widetilde{\mathfrak{N}}^{b}_{i}.
\end{eqnarray}
Furthermore, the local components of the horizontal and vertical part of second fundamental form are given respectively by 
\begin{eqnarray}\label{coefsfh}
 \stackrel{h}{\mathcal{S}^{a}_{\alpha \beta}}&=& \widetilde{\mathfrak{N}}^{a}_{i}\left(B^{i}_{\alpha \beta}+ \widetilde{\Gamma}^{i}_{jk}B^{jk}_{\alpha \beta} + \widetilde{\gamma}_{jk}^{i}B^{j}_{\alpha}H^{a}_{\beta}\mathfrak{N}_{a}^{k}\right),\\ \hbox{and} \quad
 \label{coefsfv} \stackrel{v}{\mathcal{S}^{a}_{\alpha \beta}} &=& \widetilde{\mathfrak{N}}^{a}_{i}\widetilde{\gamma}_{jk}^{i}B^{jk}_{\alpha \beta}.
\end{eqnarray}
Finally the local components of the  horizontal and vertical part of shape operator are given respectively by:
\begin{eqnarray}\label{coefsoh}
 \stackrel{h}{\mathcal{A}^{\lambda}_{a\alpha}}&=& -\left(\frac{\delta B^{k}_{a}}{\delta u^{\alpha}} + B^{j}_{a}B^{k}_{\alpha}\widetilde{\Gamma}^{i}_{jk}+ B^{j}_{a}H^{a}_{\alpha}\mathfrak{N}_{a}^{k}\widetilde{\gamma}_{jk}^{i}\right)\widetilde{B}^{\lambda}_{i},\\
 \hbox{and}\quad \label{coefsov}
  \stackrel{v}{\mathcal{A}^{\lambda}_{a\alpha}}&=& -\left(F\frac{\partial B^{k}_{a}}{\partial v^{\alpha}}+ B^{j}_{a}\widetilde{\gamma}^{i}_{jk}B^{k}_{\alpha}\right)\widetilde{B}^{\lambda}_{i}.  
\end{eqnarray}

\section{Induced and intrinsic Hashiguchi connection}\label{InduIntri}

As an application of the general theory of pullback Finsler connection developed in the previous section, we consider in the following ambient manifold $(\widetilde{M},\widetilde{F})$ endowed with \textit{Hashiguchi connection}, and study the induced one in $(M,F)$.

Now denote by $ \stackrel{H}{\widetilde{\nabla}}$ the Hashiguchi connection on the pullback bundle $\pi^{\ast}T\widetilde{M}$, given  locally by:
\begin{eqnarray}
 \stackrel{H}{\widetilde{\nabla}}_{\frac{\delta}{\delta x^{i}}}\frac{\partial}{\partial x^{j}}&=& \widetilde{\mathfrak{H}}_{ij}^{k}\frac{\partial}{\partial x^{k}},\\
 \stackrel{H}{\widetilde{\nabla}}_{F\frac{\partial}{\partial y^{i}}}\frac{\partial}{\partial x^{j}}&=& \widetilde{\mathfrak{h}}_{ij}^{k}\frac{\partial}{\partial x^{k}}.
\end{eqnarray}

Recall that $\widetilde{\mathfrak{H}}^{k}_{ij}$ and $\widetilde{\mathfrak{h}}^{k}_{ij}$ are given, respectively, by\cite{mb1}:
\begin{eqnarray}
 \widetilde{\mathfrak{H}}^{k}_{ij}&=& \widetilde{\Gamma}^{k}_{ij} + \widetilde{L}^{k}_{ij}, \\
\hbox{and}\;\;\;\widetilde{\mathfrak{h}}^{k}_{ij}&=& \widetilde{\gamma}^{k}_{ij},
\end{eqnarray}
where $ \widetilde{L}^{k}_{ij}$ are the coefficients of Landsberg tensor with respect to $\widetilde{g}$ (see\cite{AF} for more details).
\begin{lemma}\label{Lans}
 Let $(\widetilde{M}, \widetilde{F})$ be a Finsler manifold and $\widetilde{L}$ be the Landsberg tensor with respect to the pullback bundle $\pi^{\ast}T\widetilde{M}$ on $(\widetilde{M},\widetilde{F})$. Then the components of the restriction $\overline{L}$ of $\widetilde{L}$ to $\pi^{\ast}T\widetilde{M}|_{TM_{0}}$ are given locally by
 \begin{equation}
  \overline{L}_{ij\alpha}= B^{k}_{\alpha}\widetilde{L}_{ijk}-H^{a}_{\alpha}\mathfrak{N}^{k}_{a}\widetilde{\gamma}_{ijk}.\nonumber
 \end{equation}
\end{lemma}
\begin{proof}
Recall that the coefficients of Landsberg tensor $\widetilde{L}_{ijk}$ are given by the horizontal covariant derivatives of $\widetilde{g}_{ij}$ denoted 
$\widetilde{g}_{ij;k}$, more precisely  $\widetilde{L}_{ijk}:=-\frac{1}{2}\widetilde{g}_{ij;k}$ (see \cite{AF}). By lemma \ref{cnli} we  have $\frac{\delta}{\delta u^{\alpha}}\in HT\widetilde {M}_{0}|_{TM_{0}}\oplus \widetilde{\theta}^{-1}(\pi^{\ast}TM^{\bot})$, and denoting that the vertical correspondent of $\mathfrak{N}_{a}$ by $\stackrel{V}{\mathfrak{N}_{a}}:= \widetilde{\theta}^{-1}(\mathfrak{N}_{a})$, and using (\ref{ehi}) and (\ref{rsmj}), we obtain 
\begin{eqnarray}
\frac{\delta}{\delta u^{\alpha}}&=& B^{i}_{\alpha}\frac{\delta}{\delta u^{i}} + \widetilde{\mathfrak{N}}_{i}^{a}(B^{i}_{\alpha 0}+ B^{j}_{\alpha}\widetilde{N}^{i}_{j})\stackrel{V}{\mathfrak{N}_{a}}\cr 
&= &B^{i}_{\alpha}\frac{\delta}{\delta u^{i}} + H^{a}_{\alpha}\stackrel{V}{\mathfrak{N}_{a}},\nonumber
\end{eqnarray} 
where $H^{a}_{\alpha}= \widetilde{\mathfrak{N}}_{i}^{a}(B^{i}_{\alpha 0}+ B^{j}_{\alpha}\widetilde{N}^{i}_{j})$. Hence, one has
\begin{eqnarray}
 \overline{L}_{ij\alpha}&=& -\frac{1}{2}\widetilde{g}_{ij;\alpha} = -\frac{1}{2} B^{k}_{\alpha}\widetilde{g}_{ij;k} -\frac{1}{2} \mathfrak{N}_{a}^{k}H^{a}_{\alpha}\widetilde{\gamma}_{ijk} \cr
 &=& B^{k}_{\alpha}\widetilde{L}_{ijk}- \mathfrak{N}_{a}^{k}H^{a}_{\alpha}\widetilde{\gamma}_{ijk},\nonumber
 \end{eqnarray} 
 as required. 
\end{proof}
\begin{theorem}\label{coefih}
 Let $(\widetilde{M},\widetilde{F})$ be a  $(m+n)$-dimensional Finsler manifold endowed with Hashiguchi connection $\stackrel{H}{\widetilde{\nabla}}$ and $(M,F)$ be a $m$-dimensional Finsler submanifold  of $(\widetilde{M},\widetilde{F})$. Then the local coefficients of the induced Hashiguchi connection are given by the following formulas:
 \begin{eqnarray}\label{coefH}
  \mathfrak{H}^{\lambda}_{\alpha \beta}&=& \left(B^{k}_{\alpha \beta}+ B^{ij}_{\alpha\beta}\widetilde{\mathfrak{H}}^{k}_{ij}\right)\widetilde{B}^{\lambda}_{k},\\
   \hbox{and} \quad
  \label{coefh}\mathfrak{h}^{\lambda}_{\alpha \beta} &=& B^{ij}_{\alpha\beta}\widetilde{\mathfrak{h}}^{k}_{ij}\widetilde{B}^{\lambda}_{k}.
 \end{eqnarray}
 \end{theorem}
 \begin{proof}
  By lemma \ref{Lans} we obtain
$ 
\overline{\mathfrak{H}}^{k}_{i\beta}= B^{j}_{\beta}\widetilde{\mathfrak{H}}^{k}_{ij}.
$ 
Writing (\ref{coef}) for the Hashiguchi connection and replacing  $\overline{\mathfrak{H}}^{k}_{i\beta}$ by its value we obtain the relation (\ref{coefH}). One has (\ref{coefh}) in the similar way the relation (\ref{coefv}).
\end{proof}
 \begin{theorem}
  Let $(M,F)$ be a $m$-dimensional Finsler submanifold  of $(\widetilde{M},\widetilde{F})$ endowed with Hashiguchi connection $\stackrel{H}{\widetilde{\nabla}}$.
  Locally, the horizontal part $\stackrel{h}{\mathfrak{S}^{a}_{\alpha \beta}}$ and vertical part $\stackrel{v}{\mathfrak{S}^{a}_{\alpha \beta}}$  of  second fundamental form  are given by:
 \begin{equation}
  \stackrel{h}{\mathfrak{S}^{a}_{\alpha \beta}}  =   \left(B^{k}_{\alpha \beta}+ B^{ij}_{\alpha\beta}\widetilde{\mathfrak{H}}^{k}_{i j}\right)\widetilde{\mathfrak{N}}^{a}_{k} \;\;
  \hbox{and}\;\;\stackrel{v}{\mathfrak{S}^{a}_{\alpha \beta}}= B^{ij}_{\alpha\beta}\widetilde{\mathfrak{h}}^{k}_{ij}\widetilde{\mathfrak{N}}^{a}_{k},
 \end{equation}
 respectively.
\end{theorem}
\begin{proof}
 The proof follows from the relations (\ref{coefsfh}) and (\ref{coefsfv}).
\end{proof}
Now we consider the normal Hashiguchi connection $\stackrel{H}{\nabla^{\bot}}$ and set 
\begin{equation}
 \stackrel{H}{\nabla^{\bot}}_{\frac{\delta}{\delta u^{\alpha}}}\mathfrak{N}_{a}= \mathfrak{H}^{b}_{a\alpha}\mathfrak{N}_{b}\quad \hbox{and}\quad
  \stackrel{H}{\nabla^{\bot}}_{F\frac{\partial}{\partial v^{\alpha}}}\mathfrak{N}_{a}= \mathfrak{h}^{b}_{a\alpha}\mathfrak{N}_{b}.
\end{equation}
\begin{theorem}
  Let  $(M,F)$ be a $m$-dimensional Finsler submanifold  of Finsler manifold $(\widetilde{M},\widetilde{F})$, endowed with Hashiguchi connection. Then the local coefficients of the normal Hashiguchi connection on $(\pi^{\ast}TM)^{\bot}$, are given by the following formulas:
  \begin{eqnarray}
    \mathfrak{H}^{b}_{a\alpha} &=&  \left(\frac{\delta B^{k}_{a}}{\delta u^{\alpha}} + B^{j}_{a}B^{k}_{\alpha}\widetilde{\mathfrak{H}}^{i}_{jk}\right)\widetilde{\mathfrak{N}}^{b}_{i},\\
    \hbox{and} \quad 
    \mathfrak{h}^{b}_{a\alpha} &=& \left(F\frac{\partial B^{k}_{a}}{\partial v^{\alpha}}+ B^{j}_{a}B^{k}_{\alpha}\widetilde{\mathfrak{h}}^{i}_{jk}\right)\widetilde{\mathfrak{N}}^{b}_{i}.    
  \end{eqnarray}
  Moreover the local coefficients of the horizontal and vertical part of shape operator are given respectively by:
  \begin{eqnarray}
   \stackrel{h}{\mathfrak{A}^{\beta}_{a \alpha}}&=&- \left(\frac{\delta B^{k}_{a}}{\delta u^{\alpha}} + B^{j}_{a}B^{k}_{\alpha}\widetilde{\mathfrak{H}}^{i}_{jk}\right)\widetilde{B}^{\beta}_{i},\\
    \hbox{and} \quad 
     \stackrel{v}{\mathfrak{A}^{\beta}_{a \alpha}}&=&-\left(F\frac{\partial B^{k}_{a}}{\partial v^{\alpha}}+ B^{j}_{a}B^{k}_{\alpha}\widetilde{\mathfrak{h}}^{i}_{jk}\right)\widetilde{B}^{\beta}_{i}.
  \end{eqnarray} 
\end{theorem}
\begin{proof}
 Writing the relations (\ref{coefnh}), (\ref{coefnv}), (\ref{coefsoh}) and (\ref{coefsov}) for the Hashiguchi connection and using the relation $
\overline{\mathfrak{H}}^{k}_{i\beta}= B^{j}_{\beta}\widetilde{\mathfrak{H}}^{k}_{ij}$,
  the result follows.
\end{proof}
In the previous paragraph, we constructed the induced Hashiguchi connection, whose local coefficients are given by (\ref{coefH}) and (\ref{coefh}). On the other hand, on Finsler submanifold lives the intrinsic Hashiguchi connection $\stackrel{H}{\nabla^{\ast}}$. More precisely, the canonical Finsler-Ehresmann connection is locally spanned by the vector fields
\begin{equation}  \label{dins}                                                                                                                                                                                                                                                                                                                                                                        
        \frac{\delta^{\ast}}{\delta^{\ast}u^{\alpha}}= \frac{\partial}{\partial u^{\alpha}}-  \hat{N}^{\lambda}_{\alpha}\frac{\partial}{\partial v^{\lambda}}                                                                                                                                                                                                                                                                                                                                                                  \end{equation}
Then using (\ref{inef}) and taking into account (\ref{dins}) we derive that
\begin{equation}\label{reiei}
    \frac{\delta^{\ast}}{\delta^{\ast}u^{\alpha}}= \frac{\delta}{\delta u^{\alpha}}- D^{\lambda}_{\alpha}\frac{\partial}{\partial v^{\lambda}},
\end{equation}
where $D^{\lambda}_{\alpha}$ are the coefficients of deformation tensor $D$. 
Elsewhere the Hashiguchi connection is given by \cite{mb1}
\begin{align}\label{ht3}
  & 2g(\stackrel{H}{\nabla}_{X}\pi_{\ast}Y, \pi_{\ast}Z) =  X.g(\pi_{\ast}Y,\pi_{\ast}Z) + Y.g(\pi_{\ast}Z,\pi_{\ast}X)
   - Z.g(\pi_{\ast}X,\pi_{\ast}Y)\cr
   &  + g(\pi_{\ast}[X,Y], \pi_{\ast}Z) -g(\pi_{\ast}[Y,Z], \pi_{\ast}X)+ g(\pi_{\ast}[Z,X], \pi_{\ast}Y)\cr 
   &  + 2A(\theta (Z),\pi_{\ast}X, \pi_{\ast}Y)-2A(\theta(Y),\pi_{\ast}X,\pi_{\ast}Z) +2L(\pi_{\ast}X,\pi_{\ast}Y,\pi_{\ast}Z),
 \end{align} 
then replacing $X,Y$ and $Z$ from (\ref{ht3}) by $\frac{\delta}{\delta u^{\alpha}}$, $\frac{\delta}{\delta u^{\beta}}$ and $\frac{\delta}{\delta u^{\lambda}}$ respectively, we obtain
\begin{equation}\label{rcici}
 2g_{\mu \lambda}\mathfrak{H}^{\mu}_{\alpha \beta}= \frac{\delta}{\delta u^{\alpha}}g_{\beta \lambda} + \frac{\delta}{\delta u^{\beta}}g_{\alpha \lambda}- \frac{\delta}{\delta u^{\lambda}}g_{\alpha\beta} + 2 L_{\alpha \beta \lambda}.
\end{equation}
\begin{theorem}
 Let $(\widetilde{M},\widetilde{F})$ be a $(m+n)$-dimensional Finsler manifold endowed with the Hashiguchi connection $\stackrel{H}{\widetilde{\nabla}}$ and $(M,F)$ be $m$-dimensional Finsler submanifold of $(\widetilde{M},\widetilde{F})$. Then the local coefficients of the induced Hashiguchi connection $(\mathfrak{H}^{\mu}_{\alpha \beta},\mathfrak{h}^{\mu}_{\alpha \beta})$ are related to the  local coefficients of intrinsic Hashiguchi connection $(\mathfrak{H}^{\ast \mu}_{\alpha \beta},\mathfrak{h}^{\ast\mu}_{\alpha \beta})$, by the following relations:
 \begin{eqnarray}\label{inds}
 \mathfrak{H}^{\mu}_{\alpha \beta}&=&  \mathfrak{H}^{\ast \mu}_{\alpha \beta}+ D^{\tau}_{\alpha}\widetilde{B}^{\mu}_{k}\widetilde{\mathfrak{h}}^{k}_{ij}B^{ij}_{\tau\beta}+ D^{\tau}_{\beta}\widetilde{B}^{\mu}_{k}\widetilde{\mathfrak{h}}^{k}_{ij}B^{ij}_{\tau\alpha}- D^{\mu}_{\epsilon}B^{ij}_{\alpha \beta}\widetilde{\mathfrak{h}}_{ij}^{k}\widetilde{B}^{\epsilon}_{k},\\
 \mathfrak{h}^{\mu}_{\alpha \beta}&=&\mathfrak{h}^{\ast\mu}_{\alpha \beta} = B^{ij}_{\alpha\beta}\widetilde{\mathfrak{h}}^{k}_{ij}\widetilde{B}^{\mu}_{k}.
 \end{eqnarray}
\end{theorem}
\begin{proof}
 Contracting (\ref{rcici}) by $g^{\mu \lambda}$, and using (\ref{reiei}) and (\ref{coefh}), the result follows.
\end{proof}
\begin{theorem}
  Let $(\widetilde{M},\widetilde{F})$ be a $(m+n)$-dimensional Finsler manifold endowed with the Hashiguchi connection $\stackrel{H}{\widetilde{\nabla}}$ and $(M,F)$ be $m$-dimensional Finsler submanifold of $(\widetilde{M},\widetilde{F})$, Then the induced Hashiguchi connection coincides with the intrinsic Hashiguchi connection on submanifold $M$ if and only if the deformation $(0,1;1)$-tensor $D$ vanishes on $TM_{0}$. Moreover, the covariant differentiation in the direction of the horizontal correspondent of distinguished section $l^{H}$, is the same for both Hashiguchi connections on the submanifold.
\end{theorem}
\begin{proof}
 If the induced Hashiguchi connection coincides with the intrinsic one, then the Finsler-Ehresmann form $\hat{\theta}=\theta$ and by the Proposition \ref{cnl}, the 
 deformation tensor $D= 0$. Reciprocally, if the tensor $D=0$ then by the relation (\ref{inds}), the induced and intrinsic Hashiguchi connection coincides.
 The last assertion is obtained by the Lemma \ref{nD}. 
\end{proof}

\section*{Acknowledgments}
The second author would like to thank the University of KwaZulu-Natal and the DST-NRF Centre of Excellence in Mathematical and Statistics Sciences (CoE-maSS) for financial support.

 \bibliographystyle{plain}

\end{document}